\documentclass[12pt]{amsart}
 
\usepackage{amssymb, amscd, txfonts}
\usepackage{graphicx}
\usepackage[all]{xy} 
 
\numberwithin{equation}{section}

\newtheorem{theorem}{Theorem}[section]
\newtheorem{proposition}[theorem]{Proposition}
\newtheorem{lemma}[theorem]{Lemma}

\theoremstyle{definition}

\theoremstyle{remark}
\newtheorem{remark}[theorem]{Remark}
\newtheorem{claim}[theorem]{Claim}

\renewcommand{\hom}{\operatorname{Hom}}
\renewcommand{\ker}{\operatorname{Ker}}

\newcommand{\Z}{\mathbb{Z}}
\newcommand{\Q}{\mathbb{Q}}
\newcommand{\R}{\mathbb{R}}
\newcommand{\C}{\mathbb{C}}
\newcommand{\proj}{{\mathbb P}}

\newcommand{\HH}{\mathbb{H}}

\newcommand{\X}{X_{\Gamma}}
\newcommand{\Y}{Y_{\Gamma}}
\newcommand{\G}{\Gamma}
\newcommand{\D}{\mathcal{D}_{\Lambda}}
\newcommand{\DD}{\mathcal{D}_{\Lambda}^{\ast}}
\newcommand{\SL}{{\rm SL}_{2}(\mathbb{Z})}
\newcommand{\SLQ}{{\rm SL}_{2}(\mathbb{Q})}
\newcommand{\MQ}{M_{2}(\mathbb{Q})}

\newcommand{\LQ}{\Lambda_{\mathbb{Q}}}
\newcommand{\LC}{\Lambda_{\mathbb{C}}}
\newcommand{\LK}{\Lambda_{K}}
\newcommand{\LLK}{\Lambda_{K}'}
\newcommand{\LLQ}{\Lambda'_{\mathbb{Q}}}
\newcommand{\LLLQ}{\Lambda''_{\mathbb{Q}}}
\newcommand{\SpL}{{\rm Sp}(\Lambda)}
\newcommand{\SpLN}{{\rm Sp}(\Lambda, N)}
\newcommand{\SpLQ}{{\rm Sp}(\Lambda_{\mathbb{Q}})}
\newcommand{\SpLLQ}{{\rm Sp}(\Lambda'_{\mathbb{Q}})}
\newcommand{\SpLLLQ}{{\rm Sp}(\Lambda''_{\mathbb{Q}})}
\newcommand{\UL}{{\rm U}(\Lambda)}
\newcommand{\ULN}{{\rm U}(\Lambda, N)}

\newcommand{\SULK}{{\rm SU}(\Lambda_{K})}
\newcommand{\SUL}{{\rm SU}(\Lambda)}
\newcommand{\SULN}{{\rm SU}(\Lambda, N)}

\DeclareMathOperator{\tr}{tr}

\begin{document}

\title[]{Rational equivalence of cusps}
\author[]{Shouhei Ma}
\thanks{Supported by JSPS KAKENHI 15H05738 and 17K14158.} 
\address{Department~of~Mathematics, Tokyo~Institute~of~Technology, Tokyo 152-8551, Japan}
\email{ma@math.titech.ac.jp}

\begin{abstract}
We prove that two cusps of the same dimension in the Baily-Borel compactification 
of some classical series of modular varieties are linearly dependent 
in the rational Chow group of the compactification. 
This gives a higher dimensional analogue of the Manin-Drinfeld theorem. 
As a consequence, we obtain a higher dimensional generalization of 
modular units as higher Chow cycles on the modular variety. 
\end{abstract} 

\maketitle


\section{Introduction}\label{sec: intro}

The classical theorem of Manin and Drinfeld (\cite{Ma}, \cite{Dr}) asserts that 
the difference of two cusps is torsion 
in the Picard group of the modular curve for a congruence subgroup of ${\SL}$. 
This had stimulated the development of 
the theory of modular units and cuspidal class groups (see \cite{K-L}). 
The original proof used modular symbols and Hecke operators on the cohomology of the modular curve 
(\cite{Ma}, \cite{Dr}). 
Later an interpretation in terms of the mixed Hodge structure of the modular curve minus the cusps 
was also found (\cite{El}). 

Our purpose in this paper is to prove a generalization of the Manin-Drinfeld theorem  
for cusps in the Baily-Borel compactification of some higher dimensional classical modular varieties. 
In higher dimension cusps are no longer divisors, but algebraic cycles of various codimension. 
We wish to clarify their contribution to the Chow group of the Baily-Borel compactification. 

The modular varieties of our object of study are 
of the following three types: 
\begin{enumerate}
\item modular varieties of orthogonal type attached to rational quadratic forms 
of signature $(2, n)$, which have only $0$-dimensional and $1$-dimensional cusps; 
\item Siegel modular varieties attached to rational symplectic forms; and 
\item modular varieties of unitary type, including the Picard modular varieties, 
attached to Hermitian forms over imaginary quadratic fields. 
\end{enumerate}
In Cartan's classification of irreducible Hermitian symmetric domains, 
these correspond to the domains $\mathcal{D}$ of type IV, III, I respectively. 
The Baily-Borel compactification (\cite{B-B}) of the modular variety ${\G}\backslash \mathcal{D}$ 
for an arithmetic group ${\G}$ 
is obtained by adjoining rational boundary components to $\mathcal{D}$ 
and then taking quotient by ${\G}$. 
Below, by a \textit{cusp} we mean 
the closure of the image of a rational boundary component 
in the Baily-Borel compactification. 

Our main results are the following. 

\begin{theorem}[orthogonal case]\label{thm: orthogonal}
Let $\Lambda$ be an integral quadratic lattice of signature $(2, n)$, 
${\G}$ a congruence subgroup of the orthogonal group ${\rm O}^{+}(\Lambda)$, 
and ${\X}$ the Baily-Borel compactification of the modular variety defined by ${\G}$. 
Let $Z_{1}, Z_{2}$ be two cusps of ${\X}$ of the same dimension, say $k\in \{ 0, 1 \}$. 
Assume that $n\geq 3$ if $k=1$. 
Then we have ${\Q}[Z_{1}]={\Q}[Z_{2}]$ 
in the rational Chow group $CH_{k}({\X})_{{\Q}}=CH_{k}({\X})\otimes_{{\Z}}{\Q}$ of ${\X}$. 
\end{theorem}

\begin{theorem}[symplectic case]\label{thm: symplectic}
Let $\Lambda$ be an integral symplectic lattice, 
${\G}$ a congruence subgroup of the symplectic group ${\rm Sp}(\Lambda)$,  
and ${\X}$ the Satake-Baily-Borel compactification of the Siegel modular variety defined by ${\G}$. 
If $Z_{1}, Z_{2}$ are two cusps of ${\X}$ of the same dimension, say $k$, 
then ${\Q}[Z_{1}]={\Q}[Z_{2}]$ in $CH_{k}({\X})_{{\Q}}$. 
\end{theorem}

\begin{theorem}[unitary case]\label{thm: unitary}
Let $K$ be an imaginary quadratic field, 
$\Lambda$ a Hermitian lattice over $\mathcal{O}_{K}$, 
${\G}$ a congruence subgroup of the unitary group ${\rm U}(\Lambda)$, 
and ${\X}$ the Baily-Borel compactification of the modular variety defined by ${\G}$. 
If $Z_{1}, Z_{2}$ are two cusps of ${\X}$ of the same dimension, say $k$, 
then ${\Q}[Z_{1}]={\Q}[Z_{2}]$ in $CH_{k}({\X})_{{\Q}}$. 
\end{theorem}

Note that the equality ${\Q}[Z_{1}]={\Q}[Z_{2}]$ in $CH_{k}({\X})_{{\Q}}$ 
is the same as the equality $N_1[Z_1]=N_2[Z_2]$ in the integral Chow group $CH_{k}({\X})$ 
for some natural numbers $N_1, N_2$. 
When $k=0$, we must have $N_1=N_2$, 
so $[Z_{1}]-[Z_{2}]$ is torsion in $CH_{0}({\X})$. 

In the symplectic case, 
when $\Lambda$ has rank $\geq 4$, 
every finite-index subgroup of ${\rm Sp}(\Lambda)$ is a congruence subgroup 
by Mennicke \cite{Me} and Bass-Lazard-Serre \cite{B-L-S}. 
The case ${\rm rk}(\Lambda)=2$ is just the case of modular curves. 

The case $(n, k)=(2, 1)$ in the orthogonal case is indeed an exception. 
We have self products of modular curves as typical examples of ${\X}$ in $n=2$, 
for which two transversal boundary curves are not homologically equivalent. 
On the other hand, we should note that 
some consideration in the case $n=2$ is necessary for our proof for the case $n\geq3$. 

The proof of Theorems \ref{thm: orthogonal} -- \ref{thm: unitary} is based on the same simple idea. 
We connect $Z_{1}$ and $Z_{2}$ by a chain of sub modular varieties or their products, 
through the interior or the boundary, 
and use induction on the dimension of modular varieties. 
This eventually reduces the problem to the Manin-Drinfeld theorem for modular curves. 
The actual argument requires case-by-case construction 
depending on the combinatorics of rational boundary components. 
We need to argue the three cases separately,  
though the symplectic and the unitary cases are similar. 
Theorem \ref{thm: orthogonal} is proved in \S \ref{sec: orthogonal}; 
Theorem \ref{thm: symplectic} in \S \ref{sec: symplectic}; and 
Theorem \ref{thm: unitary} in \S \ref{sec: unitary}. 

In \S \ref{sec: higher modular unit}, 
as a consequence of these results, 
we associate an explicit nonzero element of the higher Chow group 
$CH_{k}({\G}\backslash \mathcal{D}, 1)_{{\Q}}$ 
of the modular variety ${\G}\backslash \mathcal{D}$ (before compactification) 
to each pair $(Z_{1}, Z_{2})$ of cusps of maximal dimension $k$. 
This gives a higher dimensional analogue of modular units from the viewpoint of algebraic cycles. 
If the span of all such higher Chow cycles on ${\G}\backslash \mathcal{D}$ 
has dimension no less than the number of maximal cusps, 
we would then obtain a nontrivial subspace of $CH_{k}({\X}, 1)_{{\Q}}$ 
for the Baily-Borel compactification ${\X}$. 


Throughout the paper $\Gamma(N)$ stands for the principal congruence subgroup 
of ${\SL}$ of level $N$, 
and $X(N)=\Gamma(N)\backslash{\HH}^{\ast}$ 
the (compactified) modular curve for $\Gamma(N)$.  
In \S \ref{sec: orthogonal} and \S \ref{sec: symplectic}, 
for a free ${\Z}$-module $\Lambda$ of finite rank, 
we denote by $\Lambda^{\vee}={\hom}_{{\Z}}(\Lambda, {\Z})$ its dual ${\Z}$-module and 
denote $\Lambda_{F}=\Lambda\otimes_{{\Z}}F$ for $F={\Q}, {\R}, {\C}$. 
For a ${\Q}$-vector space $V$ we also write 
$V^{\vee}={\hom}_{{\Q}}(V, {\Q})$ and 
$V_{F}=V\otimes_{{\Q}}F$ 
when no confusion is likely to occur.


\section{The orthogonal case}\label{sec: orthogonal}

In this section we prove Theorem \ref{thm: orthogonal}. 
We first recall orthogonal modular varieties (cf.~\cite{Sc}, \cite{Lo}). 
Let $\Lambda$ be a free ${\Z}$-module of rank $2+n$ equipped with 
a nondegenerate symmetric bilinear form 
$(\: , \: ) : \Lambda \times \Lambda \to {\Z}$ 
of signature $(2, n)$. 
Let 
\begin{equation*}
Q_{\Lambda} = \{ [{\C}\omega] \in {\proj}{\LC} \: | \: (\omega, \omega)=0 \} 
\end{equation*}
be the isotropic quadric in ${\proj}{\LC}$. 
The open set of $Q_{\Lambda}$ defined by the condition 
$(\omega, \bar{\omega})>0$ 
consists of two connected components, 
and the Hermitian symmetric domain ${\D}$ attached to $\Lambda$ is defined as 
one of them. 
This choice is equivalent to the choice of an orientation of 
a positive definite plane in $\Lambda_{{\R}}$. 

Let ${\rm O}(\Lambda)$ be the orthogonal group of $\Lambda$, 
namely the group of isomorphisms $\Lambda \to \Lambda$ preserving the quadratic form. 
We write ${\rm O}^{+}(\Lambda)$ for the subgroup of ${\rm O}(\Lambda)$ 
preserving the component ${\D}$. 
For a natural number $N$ let 
${\rm O}^{+}(\Lambda, N)<{\rm O}^{+}(\Lambda)$ 
be the kernel of the reduction map 
${\rm O}^{+}(\Lambda)\to {\rm GL}(\Lambda/N\Lambda)$. 
A subgroup ${\G}$ of ${\rm O}^{+}(\Lambda)$ is called a congruence subgroup 
if it contains ${\rm O}^{+}(\Lambda, N)$ for some level $N$. 
A typical example is the kernel of the reduction map 
${\rm O}^{+}(\Lambda) \to {\rm GL}(\Lambda^{\vee}/\Lambda)$ 
for the discriminant group $\Lambda^{\vee}/\Lambda$. 

There are two types of rational boundary components of ${\D}$, 
$0$-dimensional and $1$-dimensional components. 
$0$-dimensional components correspond to 
isotropic ${\Q}$-lines $I$ in ${\LQ}$: 
we take the point $p_{I}=[I_{{\C}}]\in Q_{\Lambda}$, 
which is in the closure of ${\D}$, 
for each such $I$. 
$1$-dimensional components correspond to 
isotropic ${\Q}$-planes $J$ in ${\LQ}$: 
we take the connected component of 
${\proj}J_{{\C}}-{\proj}J_{{\R}}\simeq {\HH}\sqcup {\HH}$, 
say ${\HH}_{J}$, that is in the closure of ${\D}$. 
The union 
\begin{equation*}
{\DD} = {\D} \sqcup \bigsqcup_{\dim J = 2} {\HH}_{J} \sqcup \bigsqcup_{\dim I =1} p_{I} 
\end{equation*}
is equipped with the Satake topology (\cite{B-B}, \cite{B-J}). 
By Baily-Borel \cite{B-B}, the quotient space 
${\X}={\G}\backslash{\DD}$ 
has the structure of a normal projective variety 
and contains ${\G}\backslash{\D}$ as a Zariski open set. 

In \S \ref{ssec: orthogonal 0dim cusp} we prove Theorem \ref{thm: orthogonal} 
for $0$-dimensional cusps, 
and in \S \ref{ssec: orthogonal 1dim cusp} for $1$-dimensional cusps. 
Throughout this section $U$ stands for the rank $2$ unimodular hyperbolic lattice 
with Gram matrix 
$\begin{pmatrix} 0 & 1 \\ 1 & 0 \end{pmatrix}$. 
The symbol $\Lambda_{1}\perp \Lambda_{2}$ stands for 
the orthogonal direct sum of two quadratic lattices (or spaces) $\Lambda_{1}, \Lambda_{2}$, 
while $\Lambda_{1}\oplus \Lambda_{2}$ just stands for 
the direct sum of $\Lambda_{1}, \Lambda_{2}$ as ${\Z}$-module (or linear space) 
and does not necessarily mean that $(\Lambda_{1}, \Lambda_{2})\equiv 0$.

\subsection{$0$-dimensional cusps}\label{ssec: orthogonal 0dim cusp}

In this subsection we prove Theorem \ref{thm: orthogonal} for $0$-dimensional cusps. 
Let $I_{1}\ne I_{2}$ be two isotropic lines in ${\LQ}$ 
and $p_{1}, p_{2}\in {\X}$ the corresponding $0$-dimensional cusps. 
We consider separately the cases where 
$(I_{1}, I_{2})\equiv 0$ or $(I_{1}, I_{2})\not\equiv 0$. 
In the former case $p_{1}$ and $p_{2}$ are joined by a boundary curve, 
while in the latter case they are joined by a modular curve through the interior of ${\X}$.

\subsubsection{The case $(I_{1}, I_{2})\equiv 0$}\label{sssec: 0dim cusp (I_{1}, I_{2})=0}
 
We first assume that $(I_{1}, I_{2})\equiv 0$. 
The direct sum $J=I_{1}\oplus I_{2}$ is an isotropic plane in ${\LQ}$. 
Let  
$
{\HH}_{J}^{\ast}={\HH}_{J}\sqcup \bigsqcup_{I\subset J} p_{I} 
$
and 
$\Gamma_{J} \subset {\rm SL}(J)$ be the image of the stabilizer of $J$ in ${\G}$. 
We have a generically injective morphism 
$f:X_{J}\to {\X}$ from the modular curve $X_{J}=\Gamma_{J}\backslash {\HH}_{J}^{\ast}$ 
whose image is the $1$-dimensional cusp associated to $J$.  

\begin{claim}\label{claim1}
$\Gamma_{J}$ is a congruence subgroup of ${\rm SL}(J_{{\Z}})$ 
where $J_{{\Z}}=J\cap \Lambda$. 
\end{claim}

\begin{proof}
There exists a rank $2$ isotropic sublattice $J'_{{\Z}}$ in ${\LQ}$ 
such that $J'_{{\Z}}\simeq (J_{{\Z}})^{\vee}$ by the pairing. 
The lattice $\Lambda_{1}= J_{{\Z}}\oplus J_{{\Z}}'$ is isometric to $U \perp U$. 
We set $\Lambda_{2}=(\Lambda_{1})^{\perp}\cap \Lambda$ 
and $\Lambda' = \Lambda_{1}\perp \Lambda_{2}$. 
Recall that ${\G}$ contains ${\rm O}^{+}(\Lambda, N)$ for some level $N$. 
Since both $\Lambda$ and $\Lambda'$ are full lattices in ${\LQ}$, 
we can find natural numbers $N_{1}, N_{2}$ such that 
\begin{equation*}
N_{1}\Lambda' \subset N\Lambda \subset \Lambda \subset N_{2}^{-1}\Lambda'. 
\end{equation*}
If we set $N'=N_{1}N_{2}$, this tells us that 
\begin{equation}\label{eqn: level argument}
{\rm O}^{+}(\Lambda', N')\subset {\rm O}^{+}(\Lambda, N) \subset {\G} 
\end{equation} 
inside
${\rm O}(\Lambda_{{\Q}})= {\rm O}(\Lambda'_{{\Q}})$. 
Now we have the embedding  
\begin{equation*}
{\SL} \simeq {\rm SL}(J_{{\Z}}) \hookrightarrow {\rm O}^{+}(\Lambda'), \qquad 
\gamma \mapsto (\gamma|_{J_{{\Z}}}) \oplus (\gamma^{\vee}|_{J'_{{\Z}}}) \oplus {\rm id}_{\Lambda_{2}}, 
\end{equation*}
whose image is contained in the stabilizer of $J$. 
Since this maps $\Gamma(N')$ into ${\rm O}^{+}(\Lambda', N')\subset {\G}$, 
we see that $\Gamma_{J}$ contains $\Gamma(N')$. 
\end{proof}

Let $q_{1}, q_{2}$ be the cusps of $X_{J}$ corresponding to $I_{1}, I_{2}$ respectively. 
By this claim we can apply the Manin-Drinfeld theorem to $X_{J}$. 
Therefore $[q_{1}]=[q_{2}]$ in $CH_{0}(X_{J})_{{\Q}}$. 
Since $f(q_{1})=p_{1}$ and $f(q_{2})=p_{2}$, we obtain 
\begin{equation*}
[p_{1}]=f_{\ast}[q_{1}]=f_{\ast}[q_{2}]=[p_{2}] 
\end{equation*}
in $CH_{0}({\X})_{{\Q}}$. 

\subsubsection{The case $(I_{1}, I_{2})\not\equiv 0$}

Next we assume that $(I_{1}, I_{2})\not\equiv 0$. 
In this case $I_{1}\oplus I_{2}$ is isometric to $U_{{\Q}}$. 
Its orthogonal complement $(I_{1}\oplus I_{2})^{\perp}$ has signature $(1, n-1)$. 
We choose a vector $v$ of positive norm from $(I_{1}\oplus I_{2})^{\perp}$ 
and put $\Lambda'_{{\Q}}=I_{1}\oplus I_{2}\oplus {\Q}v$. 
Then $\Lambda'_{{\Q}}$ has signature $(2, 1)$. 
Let $\mathcal{D}_{\Lambda'}$ be the Hermitian symmetric domain attached to $\Lambda'_{{\Q}}$. 
We have the natural inclusion $\mathcal{D}_{\Lambda'}^{\ast}\subset {\DD}$ 
which is compatible with the embedding of orthogonal groups 
\begin{equation*}
\iota : {\rm O}^{+}(\Lambda'_{{\Q}}) \hookrightarrow {\rm O}^{+}({\LQ}), \qquad 
\gamma \mapsto \gamma \oplus {\rm id}_{(\Lambda'_{{\Q}})^{\perp}}. 
\end{equation*}

\begin{claim}
There is a subgroup $\Gamma' < {\rm O}^{+}(\Lambda'_{{\Q}})$ such that 
$\iota(\Gamma')\subset {\G}$ and that 
$X'=\Gamma'\backslash \mathcal{D}_{\Lambda'}^{\ast}$ is 
naturally isomorphic to $X(N)$ for some level $N$. 
\end{claim}

\begin{proof}
Let $\Lambda_{1}=U\perp \langle 2 \rangle$. 
Then $\Lambda_{{\Q}}'$ is isometric to the scaling  
of $(\Lambda_{1})_{{\Q}}$ by some positive rational number. 
This gives natural isomorphisms 
$\mathcal{D}_{\Lambda'}^{\ast} \simeq \mathcal{D}_{\Lambda_{1}}^{\ast}$ and 
${\rm O}^{+}(\Lambda_{{\Q}}') \simeq {\rm O}^{+}((\Lambda_{1})_{{\Q}})$. 
The group ${\rm O}^{+}((\Lambda_{1})_{{\Q}})$ is related to ${\rm SL}_{2}({\Q})$ 
by the following well-known construction (cf.~\cite{M-R} \S 2.4). 
Let $V\subset M_{2}({\Q})$ be the space of $2\times 2$ matrices with trace $0$, 
equipped with the symmetric form $(A, B)={\tr}(AB)$. 
Then $V\cap M_{2}({\Z})$ is isometric to $\Lambda_{1}$. 
By conjugation ${\rm SL}_{2}({\Q})$ acts on $V$. 
This defines a homomorphism 
\begin{equation*}
\varphi : {\rm SL}_{2}({\Q}) \to {\rm O}^{+}(V)={\rm O}^{+}((\Lambda_{1})_{{\Q}}) 
\end{equation*}
with ${\ker}(\varphi)=\{ \pm I \}$. 
(We have ${\rm Im}(\varphi)={\rm SO}^{+}(V)$, 
but we do not need this fact.) 
It is readily checked that 
$\varphi(\Gamma(N))\subset {\rm O}^{+}(\Lambda_{1}, N)$ 
for every level $N$. 
Furthermore, $\varphi$ is compatible with the Veronese isomorphism 
\begin{equation*}
{\HH}^{\ast} \to \mathcal{D}_{\Lambda_{1}}^{\ast}, \qquad 
\tau \mapsto e + \tau v_{0} -\tau^{2}f, 
\end{equation*}
where $e, f$ are the standard basis of $U$ and $v_{0}$ is a generator of $\langle 2 \rangle$. 
Now by the same argument as \eqref{eqn: level argument}, 
there exists a level $N$ such that the embedding $\iota$ maps 
${\rm O}^{+}(\Lambda_{1}, N)$ into ${\G}$. 
This proves our claim. 
\end{proof}

Let $q_{1}, q_{2}$ be the cusps of $X'$ corresponding to 
the isotropic lines $I_{1}, I_{2}$ of $\Lambda'_{{\Q}}$. 
By this claim we have a finite morphism 
$f:X'\to {\X}$ 
which sends $q_{1}$ to $p_{1}$ and $q_{2}$ to $p_{2}$. 
By the Manin-Drinfeld theorem for $X'$ we have 
$[q_{1}]=[q_{2}]$ in $CH_{0}(X')_{{\Q}}$. 
Applying $f_{\ast}$, 
we obtain $[p_{1}]=[p_{2}]$ in $CH_{0}({\X})_{{\Q}}$. 
This finishes the proof of Theorem \ref{thm: orthogonal} for $0$-dimensional cusps.

\begin{remark}\label{remark1}
If $\Lambda$ has Witt index $2$, 
$(I_{1}\oplus I_{2})^{\perp}$ contains an isotropic line, say $I_{3}$. 
Then we could also apply the result of \S \ref{sssec: 0dim cusp (I_{1}, I_{2})=0} 
to $I_{1}$ vs $I_{3}$ and to $I_{3}$ vs $I_{2}$, 
thus obtaining $[p_{1}]=[p_{2}]$ via $I_{3}$. 
Together with the case of \S \ref{sssec: 0dim cusp (I_{1}, I_{2})=0}, 
this shows that  when $X_{\Gamma}$ contains at least one $1$-dimensional cusp, 
then any two $0$-dimensional cusps can be connected by a chain of $1$-dimensional cusps 
of length $\leq 2$ 
which provides their rational equivalence. 
\end{remark}

\subsection{$1$-dimensional cusps}\label{ssec: orthogonal 1dim cusp}

In this subsection we prove Theorem \ref{thm: orthogonal} for $1$-dimensional cusps.

\subsubsection{Preliminaries in $n=2$}\label{sssec: n=2}

Although the case $n=2$ is not included in Theorem \ref{thm: orthogonal} for $1$-dimensional cusps, 
we need to study a specific example in $n=2$  
as preliminaries for the proof for the case $n\geq3$. 
We consider the lattice $2U=U\perp U$. 
Let $e_{1}, f_{1}$ be the standard basis of the first copy of $U$, 
and $e_{2}, f_{2}$ be that of the second $U$. 
Let $J_{1}'={\Q}e_{2}\oplus{\Q}e_{1}$ and $J_{2}'={\Q}f_{2}\oplus{\Q}f_{1}$, 
which are isotropic planes in $2U_{{\Q}}$. 
We take an arbitrary natural number $N$ and consider the modular surface 
$S(N)={\rm O}^{+}(2U, N)\backslash \mathcal{D}_{2U}^{\ast}$. 
Let $C_{1}, C_{2}$ be the boundary curves of $S(N)$ 
associated to $J_{1}', J_{2}'$ respectively. 

\begin{lemma}\label{lem: n=2}
We have ${\Q}[C_{1}]={\Q}[C_{2}]$ in $CH_{1}(S(N))_{{\Q}}$. 
\end{lemma}

\begin{proof}
Recall that we have the Segre isomorphism 
\begin{equation}\label{eqn: Segre}
{\HH}\times {\HH} \to \mathcal{D}_{2U}, \qquad 
(\tau_{1}, \tau_{2})\mapsto e_{1}-\tau_{1}\tau_{2}f_{1}+\tau_{1}e_{2}+\tau_{2}f_{2}. 
\end{equation}
This extends to 
${\HH}^{\ast}\times {\HH}^{\ast} \to \mathcal{D}_{2U}^{\ast}$, 
and maps the boundary components 
${\HH}\times (\tau_{2}=0), {\HH}\times (\tau_{2}=i\infty)$ 
of ${\HH}^{\ast}\times {\HH}^{\ast}$ 
to the boundary components 
${\HH}_{J_{1}'}, {\HH}_{J_{2}'}$  of $\mathcal{D}_{2U}^{\ast}$ respectively. 

Let $J_{3}'={\Q}f_{2}\oplus{\Q}e_{1}$ and $J_{4}'={\Q}e_{2}\oplus{\Q}f_{1}$. 
By the pairing we identify $J_{2}'\simeq (J_{1}')^{\vee}$ and $J_{4}'\simeq (J_{3}')^{\vee}$. 
Then we define an embedding 
\begin{equation*}
{\rm SL}_{2}({\Q})\times {\rm SL}_{2}({\Q}) =  
{\rm SL}(J_{1}')\times {\rm SL}(J_{3}') \hookrightarrow 
{\rm O}^{+}(2U_{{\Q}}) 
\end{equation*}
by sending 
$\gamma_{1}\in {\rm SL}(J_{1}')$ to 
$(\gamma_{1}|_{J_{1}'})\oplus (\gamma_{1}^{\vee}|_{J_{2}'})$ 
and $\gamma_{3}\in {\rm SL}(J_{3}')$ to 
$(\gamma_{3}|_{J_{3}'})\oplus (\gamma_{3}^{\vee}|_{J_{4}'})$. 
This embedding of groups is compatible with the isomorphism \eqref{eqn: Segre} of domains, 
and it maps $\Gamma(N)\times \Gamma(N)$ into ${\rm O}^{+}(2U, N)$. 
We thus obtain a finite morphism 
$f: X(N)\times X(N)\to S(N)$ which maps the boundary curves 
\begin{equation*}
C_{1}'=X(N)\times (\tau_{2}=0), \quad  
C_{2}'=X(N)\times (\tau_{2}=i\infty) 
\end{equation*}
of $X(N)\times X(N)$ onto $C_{1}, C_{2}$ respectively.  
By the Manin-Drinfeld theorem for the second copy of $X(N)$, 
we have $[C_{1}']=[C_{2}']$ in $CH_{1}(X(N)\times X(N))_{{\Q}}$. 
Applying $f_{\ast}$, we obtain the assertion. 
\end{proof}

\subsubsection{The case $J_{1}\cap J_{2}=\{ 0 \}$}\label{sssec: I_1capI2=0}

We go back to the proof of Theorem \ref{thm: orthogonal}. 
Let $\Lambda$ have signature $(2, n)$ with $n\geq 3$. 
Let $J_{1}\ne J_{2}$ be two isotropic planes in ${\LQ}$ 
and $Z_1, Z_2\subset {\X}$ the corresponding $1$-dimensional cusps.  
We first consider the case where $J_{1}\cap J_{2}= \{ 0\}$. 
In this case the pairing between $J_{1}$ and $J_{2}$ is perfect 
because $J_{i}^{\perp}/J_{i}$ is negative definite. 
The direct sum $\Lambda_{{\Q}}'=J_{1}\oplus J_{2}$ is isometric to $2U_{{\Q}}$. 
We can take an isometry $2U_{{\Q}}\to \Lambda_{{\Q}}'$ which maps 
$J_{1}', J_{2}'$ to $J_{1}, J_{2}$ respectively. 
This gives an embedding of orthogonal groups 
\begin{equation}\label{eqn: embed O(2U)}
{\rm O}^{+}(2U_{{\Q}}) \simeq {\rm O}^{+}(\Lambda_{{\Q}}') 
\hookrightarrow {\rm O}^{+}(\Lambda_{{\Q}}), \qquad 
\gamma \mapsto \gamma\oplus {\rm id}_{(\Lambda_{{\Q}}')^{\perp}}, 
\end{equation}
which is compatible with the embedding  
$\mathcal{D}_{2U} \simeq \mathcal{D}_{\Lambda'} \subset {\D}$ 
of domains. 
By the same argument as \eqref{eqn: level argument}, 
we can find a level $N$ such that the embedding \eqref{eqn: embed O(2U)} 
maps ${\rm O}^{+}(2U, N)$ into ${\G}$. 
We thus obtain a finite morphism 
$f:S(N)\to {\X}$ with $f(C_{1})=Z_{1}$ and $f(C_{2})=Z_{2}$. 
Sending the equality 
${\Q}[C_{1}]={\Q}[C_{2}]$ 
of Lemma \ref{lem: n=2} by $f_{\ast}$, 
we obtain ${\Q}[Z_{1}]={\Q}[Z_{2}]$ in $CH_{1}({\X})_{{\Q}}$.

\subsubsection{The case $J_{1}\cap J_{2}\ne\{ 0 \}$}\label{sssec: I_1capI2not0}

We next consider the case where $J_{1}\cap J_{2}\ne \{ 0 \}$. 
Let $I=J_{1}\cap J_{2}$ and choose splittings 
$J_{1}=I\oplus I_{1}$ and $J_{2}=I\oplus I_{2}$. 
Since $(I_{1}, I_{2})\not\equiv 0$, we have $I_{1}\oplus I_{2}\simeq U_{{\Q}}$. 
Let ${\LLQ}=I_{1}\oplus I_{2}$ and ${\LLLQ}=({\LLQ})^{\perp}$. 
Then ${\LLLQ}$ has signature $(1, n-1)$. 
Since $n-1\geq2$ and ${\LLLQ}$ contains at least one isotropic line $I$, 
then ${\LLLQ}$ contains infinitely many isotropic lines. 
We can choose isotropic lines $I_{3}, I_{4}$ in ${\LLLQ}$ 
such that $I, I_{3}, I_{4}$ are linearly independent. 
Put 
$J_{3}=I_{4}\oplus I_{2}$ and $J_{4}=I_{3}\oplus I_{1}$. 
Then $J_{3}, J_{4}$ are isotropic of dimension $2$ 
and we have 
\begin{equation*}
J_{1}\cap J_{3}= \{ 0 \}, \quad 
J_{3}\cap J_{4}= \{ 0 \}, \quad 
J_{4}\cap J_{2}= \{ 0 \}. 
\end{equation*}
If $Z_{i}\subset {\X}$ is the $1$-dimensional cusp associated to $J_{i}$, 
we can apply the result of \S \ref{sssec: I_1capI2=0} successively 
and obtain 
\begin{equation*}
{\Q}[Z_{1}] = {\Q}[Z_{3}] = {\Q}[Z_{4}] = {\Q}[Z_{2}] 
\end{equation*}
in $CH_{1}({\X})_{{\Q}}$. 
This finishes the proof of Theorem \ref{thm: orthogonal} for $1$-dimensional cusps.


\section{The symplectic case}\label{sec: symplectic}

In this section we prove Theorem \ref{thm: symplectic}. 
We first recall Siegel modular varieties (cf.~\cite{H-K-W}, \cite{Lo}). 
Let $\Lambda$ be a free ${\Z}$-module of rank $2g$ equipped with 
a nondegenerate symplectic form 
$( \: , \: ):\Lambda \times \Lambda \to {\Z}$. 
Let ${\SpL}$ be the symplectic group of $\Lambda$, namely 
the group of isomorphisms $\Lambda \to \Lambda$ preserving the symplectic form. 
For a natural number $N$ we write ${\SpLN}$ for the kernel of the reduction map 
${\SpL}\to {\rm GL}(\Lambda/N\Lambda)$. 
A subgroup ${\G}$ of ${\SpL}$ is called a congruence subgroup 
if it contains ${\SpLN}$ for some level $N$. 
When $g\geq2$, every finite-index subgroup of ${\SpL}$ 
is a congruence subgroup (\cite{Me}, \cite{B-L-S}). 

Let 
\begin{equation*}
LG_{\Lambda}=\{ [V] \in G(g, {\LC}) \: | \: ( \: , \:)|_{V}\equiv 0 \} 
\end{equation*}
be the Lagrangian Grassmannian parametrizing 
$g$-dimensional (= maximal) isotropic ${\C}$-subspaces of ${\LC}$. 
The Hermitian symmetric domain attached to $\Lambda$ is defined as 
the open locus ${\D}\subset LG_{\Lambda}$ of those $[V]$ 
such that the Hermitian form $i(\cdot, \bar{\cdot})|_{V}$ on $V$ is positive definite. 

Rational boundary components of ${\D}$ correspond to 
isotropic ${\Q}$-subspaces $I$ of ${\LQ}$. 
To each such $I$ we associate the locus 
$\mathcal{D}_{I}\subset LG_{\Lambda}$ 
of those $[V]$ which contains $I$ 
and for which $i(\cdot, \bar{\cdot})|_{V}$ is positive semidefinite with kernel $I_{{\C}}$. 
If we consider the rational symplectic space 
${\LLQ}=I^{\perp}/I$, 
then $\mathcal{D}_{I}$ is canonically isomorphic to 
the Hermitian symmetric domain $\mathcal{D}_{\Lambda'}$ attached to ${\LLQ}$ 
by mapping $[V]\in\mathcal{D}_{I}$ to $[V/I_{{\C}}]\in \mathcal{D}_{\Lambda'}$. 
The union 
\begin{equation*}
{\DD} = {\D} \sqcup \bigsqcup_{I\subset {\LQ}} \mathcal{D}_{I}  
\end{equation*}
is equipped with the Satake topology (\cite{B-B}, \cite{B-J}, \cite{H-K-W}). 
By Baily-Borel \cite{B-B}, the quotient space 
${\X}={\G}\backslash{\DD}$ 
has the structure of a normal projective variety 
and contains ${\G}\backslash{\D}$ as a Zariski open set. 

Theorem \ref{thm: symplectic} is proved by induction on $g$. 
The case $g=1$ follows from the Manin-Drinfeld theorem. 
Let $g\geq 2$. 
Assume that the theorem is proved for every congruence subgroup 
of ${\rm Sp}(\Lambda')$ for every symplectic lattice $\Lambda'$ of rank $<2g$. 
We then prove the theorem for ${\G}<{\SpL}$ with $\Lambda$ rank $2g$. 
 
Let $I_{1}\ne I_{2}$ be two isotropic ${\Q}$-subspaces of ${\LQ}$ of the same dimension, say $g'$,  
and $Z_{1}, Z_{2}\subset {\X}$ the corresponding cusps. 
If we write $g''=g-g'$, then $Z_{i}$ has dimension $k=g''(g''+1)/2$. 
We consider the following three cases separately: 
\begin{enumerate}
\item $I_{1}\cap I_{2}\ne \{ 0 \}$;  
\item the pairing between $I_{1}$ and $I_{2}$ is perfect;  
\item $I_{1}\cap I_{2}= \{ 0 \}$ but the pairing between $I_{1}$ and $I_{2}$ is not perfect. 
\end{enumerate} 
The case (1) is studied in \S \ref{ssec: symplectic case 1} 
where $Z_{1}$ and $Z_{2}$ are joined by a modular variety in the boundary. 
The case (2) is studied in \S \ref{ssec: symplectic case 2} 
where $Z_{1}$ and $Z_{2}$ are joined by a product of two modular varieties (when $g'=1$)  
or by a chain of boundary modular varieties (when $g'>1$). 
The remaining case (3) is considered in \S \ref{ssec: symplectic case 3} 
where we combine the results of (1) and (2).

\subsection{The case $I_{1}\cap I_{2}\ne \{ 0\}$}\label{ssec: symplectic case 1}

Assume that $I_{1}\cap I_{2}\ne \{ 0\}$. 
Let $I=I_{1}\cap I_{2}$. 
In this case $\mathcal{D}_{I_{1}}, \mathcal{D}_{I_{2}}$ are in the boundary of $\mathcal{D}_{I}$. 
We set ${\LLQ}=I^{\perp}/I$, $I_{1}'=I_{1}/I$ and $I_{2}'=I_{2}/I$. 
Then $I_{1}', I_{2}'$ are isotropic subspaces of ${\LLQ}$.  
The isomorphism $\mathcal{D}_{I}\to \mathcal{D}_{\Lambda'}$ 
extends to $\mathcal{D}_{I}^{\ast}\to \mathcal{D}_{\Lambda'}^{\ast}$ 
and maps $\mathcal{D}_{I_{i}}$ to $\mathcal{D}_{I_{i}'}$. 
The stabilizer of $I$ in ${\G}$ acts on ${\LLQ}$ naturally. 
Let $\Gamma_{I}<{\SpLLQ}$ be its image in ${\SpLLQ}$. 
By a similar argument as Claim \ref{claim1}, 
$\Gamma_{I}$ is a congruence subgroup of ${\rm Sp}(\Lambda')$ 
for some lattice $\Lambda'\subset {\LLQ}$. 
If we put $X_{I}=\Gamma_{I}\backslash \mathcal{D}_{\Lambda'}^{\ast}$, 
we have a generically injective morphism 
$f:X_{I} \to {\X}$ 
onto the $I$-cusp. 

Let $Z_{1}', Z_{2}'\subset X_{I}$ be the cusps of $X_{I}$ 
corresponding to $I_{1}', I_{2}'\subset {\LLQ}$ respectively. 
By our hypothesis of induction, we have 
${\Q}[Z_{1}']={\Q}[Z_{2}']$ in $CH_{k}(X_{I})_{{\Q}}$. 
Since $f(Z_{i}')=Z_{i}$, 
applying $f_{\ast}$ gives 
${\Q}[Z_{1}]={\Q}[Z_{2}]$ in $CH_{k}({\X})_{{\Q}}$. 

\subsection{The case $(I_{1}, I_{2})$ perfect}\label{ssec: symplectic case 2}

Next we consider the case where the pairing between $I_{1}$ and $I_{2}$ is perfect. 
We shall distinguish the case $g'>1$ and the case $g'=1$ (i.e., top dimensional cusps).

\subsubsection{The case $g'>1$}\label{sssec: symplectic perfect k=0}

First let $g'>1$. 
We can choose a proper subspace $J_{1}\ne \{ 0\}$ of $I_{1}$. 
We put 
$J_{2}=J_{1}^{\perp}\cap I_{2}$ and $I_{3}=J_{1}\oplus J_{2}$. 
Then $I_{3}$ is isotropic of dimension $g'$. 
By construction we have 
$I_{1}\cap I_{3}\ne \{ 0 \}$ and $I_{3}\cap I_{2}\ne \{ 0 \}$. 
Therefore we can apply the result of \S \ref{ssec: symplectic case 1} 
to $I_{1}$ vs $I_{3}$ and to $I_{3}$ vs $I_{2}$. 
If $Z_{3}$ is the cusp of ${\X}$ associated to $I_{3}$, 
this gives 
${\Q}[Z_{1}]={\Q}[Z_{3}]={\Q}[Z_{2}]$ 
in $CH_{k}({\X})_{{\Q}}$.

\subsubsection{The case $g'=1$}\label{sssec: symplectic perfect k>0}

Next let $g'=1$. 
We set ${\LLQ}=I_{1}\oplus I_{2}$, 
which is a nondegenerate symplectic space of dimension $2$. 
Then ${\LLLQ}:=({\LLQ})^{\perp}$ is also nondegenerate of dimension $2g-2$ 
and we have ${\LQ}={\LLQ}\perp{\LLLQ}$. 
Let $\mathcal{D}_{\Lambda'}$, $\mathcal{D}_{\Lambda''}$ 
be the Hermitian symmetric domains attached to ${\LLQ}, {\LLLQ}$ respectively. 
We have the embedding of domains 
\begin{equation}\label{eqn: symplectic case 2 domain embed}
\mathcal{D}_{\Lambda'} \times \mathcal{D}_{\Lambda''} 
\hookrightarrow {\D}, \qquad 
(V', V'') \mapsto V'\oplus V''. 
\end{equation}
This is compatible with the embedding of groups 
\begin{equation}\label{eqn: symplectic case 2 group embed}
{\SpLLQ} \times {\SpLLLQ} 
\hookrightarrow {\SpLQ}, \qquad 
(\gamma', \gamma'') \mapsto \gamma'\oplus \gamma''. 
\end{equation}
The isotropic lines $I_{1}, I_{2}$ in ${\LLQ}$ 
correspond to the rational boundary points 
$[(I_{1})_{{\C}}], [(I_{2})_{{\C}}]$ of $\mathcal{D}_{\Lambda'}\simeq {\HH}$. 
Then \eqref{eqn: symplectic case 2 domain embed} extends to 
$\mathcal{D}_{\Lambda'}^{\ast} \times \mathcal{D}_{\Lambda''}^{\ast} \hookrightarrow {\DD}$ 
and maps $[(I_{i})_{{\C}}]\times \mathcal{D}_{\Lambda''}$ to $\mathcal{D}_{I_{i}}$. 

We take some full lattices 
$\Lambda'\subset {\LLQ}$ and $\Lambda''\subset {\LLLQ}$. 
By the same argument as \eqref{eqn: level argument}, 
we can find a level $N$ such that 
\eqref{eqn: symplectic case 2 group embed} maps 
${\rm Sp}(\Lambda', N)\times {\rm Sp}(\Lambda'', N)$ 
into ${\G}$. 
If we put 
$X'={\rm Sp}(\Lambda', N)\backslash \mathcal{D}_{\Lambda'}^{\ast}$ and  
$X''={\rm Sp}(\Lambda'', N)\backslash \mathcal{D}_{\Lambda''}^{\ast}$, 
we thus obtain a finite morphism 
$f:X'\times X'' \to {\X}$. 
Let $p_{1}, p_{2}$ be the cusps of the modular curve $X'$ 
corresponding $I_{1}, I_{2}\subset {\LLQ}$ respectively. 
If we set  
\begin{equation*}
Z_{i}' = p_{i}\times X'' \subset X' \times X'', 
\end{equation*}
the above consideration shows that $f(Z_{i}')=Z_{i}$. 

We have 
$[p_{1}]=[p_{2}]$ in $CH_{0}(X')_{{\Q}}$ 
by the Manin-Drinfeld theorem. 
Taking pullback by $X'\times X''\to X'$, 
we obtain 
$[Z_{1}']=[Z_{2}']$ in $CH_{k}(X'\times X'')_{{\Q}}$. 
Then, taking pushforward by $f$, 
we obtain 
${\Q}[Z_{1}]={\Q}[Z_{2}]$ in $CH_{k}({\X})_{{\Q}}$.  

\subsection{The remaining case}\label{ssec: symplectic case 3}

Finally we consider the remaining case, 
namely $I_{1}\cap I_{2} = \{ 0 \}$ 
but the pairing between $I_{1}$ and $I_{2}$ is not perfect. 
Let $J_{1}\subset I_{1}$ and $J_{2}\subset I_{2}$ 
be the kernels of the pairing between $I_{1}$ and $I_{2}$. 
We choose splittings 
$I_{1}=J_{1}\oplus K_{1}$ and $I_{2}=J_{2}\oplus K_{2}$. 
Then $\dim J_{1}=\dim J_{2}$ and 
the pairing between $K_{1}$ and $K_{2}$ is perfect. 
(We may have $K_{i}=\{ 0 \}$. 
This is the case, e.g., when $g'=1$.) 
We set ${\LLQ}=K_{1}\oplus K_{2}$ and ${\LLLQ}=({\LLQ})^{\perp}$, 
which are nondegenerate subspaces of ${\LQ}$ with 
${\LQ}={\LLQ}\perp{\LLLQ}$. 
By definition $J_{1}$ and $J_{2}$ are isotropic subspaces of ${\LLLQ}$ 
with $J_{1}\cap J_{2} = \{ 0 \}$ and $(J_{1}, J_{2})\equiv 0$. 
We can take another isotropic subspace $J_{0}$ of ${\LLLQ}$ 
of the same dimension as $J_{1}, J_{2}$ such that 
the pairings $(J_{0}, J_{1})$ and $(J_{0}, J_{2})$ are perfect. 
We set 
$I_{3}=J_{0}\oplus K_{2}$ and $I_{4}=J_{0}\oplus K_{1}$. 
Then $I_{3}, I_{4}$ are isotropic subspaces of ${\LQ}$ 
of the same dimension as $I_{1}, I_{2}$. 
By construction 
the pairings $(I_{1}, I_{3})$ and $(I_{2}, I_{4})$ are perfect, 
and we have $I_{3}\cap I_{4}\ne \{ 0 \}$. 
Then we can apply 
the result of \S \ref{ssec: symplectic case 2} to $I_{1}$ vs $I_{3}$ and to $I_{2}$ vs $I_{4}$, 
and when $K_{i} \ne \{ 0 \}$ the result of \S \ref{ssec: symplectic case 1} to $I_{3}$ vs $I_{4}$. 
(When $K_{i}=\{ 0 \}$, so that $I_{3}=I_{4}$, 
the latter process is skipped.)  
If $Z_{3}, Z_{4}$ are the cusps of ${\X}$ associated to $I_{3}, I_{4}$ respectively, 
this shows that 
\begin{equation*}
{\Q}[Z_{1}] = {\Q}[Z_{3}] = {\Q}[Z_{4}] = {\Q}[Z_{2}] 
\end{equation*} 
in $CH_{k}({\X})_{{\Q}}$. 
This completes the proof of Theorem \ref{thm: symplectic}.

\begin{remark}\label{remark2}
Summing up the argument in the case $g'>1$, 
we see that if $Z_{1}$ and $Z_{2}$ are not top dimensional, 
we can obtain their rational equivalence through 
a chain of higher dimensional cusps of length $\leq 5$. 
\end{remark}


\section{The unitary case}\label{sec: unitary}

In this section we prove Theorem \ref{thm: unitary}. 
We first recall modular varieties of unitary type (cf.~\cite{Ho}, \cite{Lo}). 
Let $K={\Q}(\sqrt{-D})$ be an imaginary quadratic field with 
$R=\mathcal{O}_{K}$ its ring of integers 
(or more generally an order in $K$). 
By a Hermitian lattice over $R$ 
we mean a finitely generated torsion-free $R$-module $\Lambda$ 
equipped with a nondegenerate Hermitian form 
$( \: , \: ): \Lambda \times \Lambda \to R$. 
We denote  
${\LK}=\Lambda\otimes_{R}K$ 
and 
${\LC}=\Lambda\otimes_{R}{\C}$, 
which are Hermitian spaces over $K, {\C}$ respectively 
and in which $\Lambda$ is naturally embedded. 
We may assume without loss of generality 
that the signature $(p, q)$ of $\Lambda$ satisfies $p\leq q$. 

Let ${\UL}$ be the unitary group of $\Lambda$, 
namely the group of $R$-linear isomorphisms 
$\Lambda \to \Lambda$ preserving the Hermitian form. 
This is the same as 
$K$-linear isomorphisms 
${\LK}\to {\LK}$ preserving the lattice $\Lambda$ and the Hermitian form. 
We write ${\rm SU}(\Lambda)$ for the subgroup of ${\UL}$ of determinant $1$. 
For a natural number $N$ we write ${\ULN}$ for the kernel of the reduction map 
${\UL}\to {\rm GL}(\Lambda/N\Lambda)$. 
A subgroup ${\G}$ of ${\UL}$ is called a congruence subgroup 
if it contains ${\ULN}$ for some level $N$. 

Let $G_{\Lambda}=G(p, {\LC})$ be the Grassmannian parametrizing 
$p$-dimensional ${\C}$-linear subspaces of ${\LC}$. 
The Hermitian symmetric domain ${\D}$ attached to $\Lambda$ is 
defined as the open locus 
\begin{equation*}
{\D} = \{ [V] \in G_{\Lambda} \: | \: ( \: , \:)|_{V}>0 \} 
\end{equation*}
of subspaces $V$ to which restriction of the Hermitian form is positive definite. 
When $p=0$, this is one point; 
when $p=1$, this is a ball in ${\proj}{\LC}\simeq {\proj}^{q}$. 

Rational boundary components of ${\D}$ correspond to 
isotropic $K$-subspaces $I$ of ${\LK}$. 
For each such $I$ we associate the locus 
$\mathcal{D}_{I}\subset G_{\Lambda}$ of those $V$ 
which contains $I$ and for which 
$( \: , \:)|_{V}$ is positive semidefinite with kernel $I_{{\C}}$.  
If we consider ${\LLK}=I^{\perp}/I$, 
this is a nondegenerate $K$-Hermitian space of signature 
$(p-r, q-r)$ where $r=\dim _{K}I$, 
and $\mathcal{D}_{I}$ is naturally isomorphic to the Hermitian symmetric domain 
$\mathcal{D}_{\Lambda'}$ attached to ${\LLK}$ by sending $[V]\in \mathcal{D}_{I}$ to $[V/I_{{\C}}]$. 
The union 
\begin{equation*}
{\DD} = {\D} \sqcup \bigsqcup_{I\subset {\LK}} \mathcal{D}_{I}  
\end{equation*}
is equipped with the Satake topology (\cite{B-B}, \cite{B-J}). 
By Baily-Borel \cite{B-B}, the quotient space 
${\X}={\G}\backslash{\DD}$ 
has the structure of a normal projective variety 
and contains ${\G}\backslash{\D}$ as a Zariski open set. 

The proof of Theorem \ref{thm: unitary} proceeds by induction on $q$. 
The case $q=1$ is the Manin-Drinfeld theorem: 
we explain this in \S \ref{ssec: q=1 unitary}. 
The inductive argument is done in \S \ref{ssec: induction unitary}. 
Since this is similar to the symplectic case, we will be brief in \S \ref{ssec: induction unitary}.

\subsection{On the case $q=1$}\label{ssec: q=1 unitary}

Let $q=1$. 
Then $r=p=q=1$, 
so ${\LK}$ is the (unique) $K$-Hermitian space of signature $(1, 1)$ 
containing an isotropic vector, 
and ${\D}$ is the unit disc in ${\proj}{\LC}\simeq {\proj}^{1}$.  
The group ${\SULK}$ is naturally isomorphic to ${\SLQ}$, 
and $\Gamma\cap {\SUL}$ is mapped to a conjugate of 
a congruence subgroup of ${\SL}$ under this isomorphism. 
This is a classical fact, but since we could not find a suitable reference 
for the second assertion, 
we supplement below a self-contained account for the convenience of the reader. 
Theorem \ref{thm: unitary} in the case $q=1$ then follows from the Manin-Drinfeld theorem, 
because we have a natural finite morphism from 
$X_{\Gamma\cap{\SUL}}$ to ${\X}$. 

We embed $K={\Q}(\sqrt{-D})$ into the matrix algebra ${\MQ}$ 
by sending $\sqrt{-D}$ to 
$J_{D}= \begin{pmatrix} 0 & -D \\ 1 & 0 \end{pmatrix}$. 
Left multiplication by $J_{D}$ makes ${\MQ}$ 
a $2$-dimensional $K$-linear space. 
We have a $K$-Hermitian form on ${\MQ}$ defined by 
\begin{equation*}
(A, B) = {\tr}(AB^{\ast}) + \sqrt{-D}^{-1}{\tr}(J_{D}AB^{\ast}), \qquad 
A, B\in {\MQ}, 
\end{equation*}
where for 
$B= \begin{pmatrix} a & b \\ c & d \end{pmatrix}$ 
we write 
$B^{\ast}= \begin{pmatrix} d & -b \\ -c & a \end{pmatrix}$. 
We denote ${\LK}={\MQ}$ 
when we want to stress this $K$-Hermitian structure. 
Then ${\LK}$ has signature $(1, 1)$ and contains an isotropic vector, e.g., 
$\begin{pmatrix} 1 & 0 \\ 0 & 0 \end{pmatrix}$. 
Right multiplication by ${\SLQ}$ on ${\MQ}$ is $K$-linear 
and preserves this Hermitian form. 
This defines a homomorphism 
\begin{equation}\label{eqn: SU(1,1) SL(2)}
{\SLQ}\to {\SULK} 
\end{equation}
which in fact is an isomorphism (see e.g., \cite{Sh1} \S 2). 

Let $\Lambda \subset {\LK}$ be a full $R$-lattice. 
We shall show that for every level $N$ 
the image of ${\SULN}={\SULK}\cap {\ULN}$ by \eqref{eqn: SU(1,1) SL(2)} 
is conjugate to a congruence subgroup of ${\SL}$. 
Let 
\begin{equation*}
\mathcal{O}=\{ X\in {\MQ} \: | \: \Lambda X \subset \Lambda \}. 
\end{equation*}
This is an order in ${\MQ}$ (see \cite{M-R} \S 2.2). 
Then 
${\SUL}=\mathcal{O}^{1}$, 
where for any subset $\mathcal{S}$ of ${\MQ}$ we write 
$\mathcal{S}^{1}=\mathcal{S}\cap {\SLQ}$. 
Take a maximal order $\mathcal{O}_{max}$ of ${\MQ}$ containing $\mathcal{O}$. 
Since $\mathcal{O}$ is of finite index in $\mathcal{O}_{max}$, 
there exists a natural number $N_{0}$ such that 
$N_{0}\mathcal{O}_{max}\subset \mathcal{O}$. 
Therefore 
\begin{equation*}
I+NN_{0}\mathcal{O}_{max} \subset 
I+N\mathcal{O} \subset 
\mathcal{O} \subset 
\mathcal{O}_{max}. 
\end{equation*}
Since 
$(I+N\mathcal{O})^{1}\subset {\SULN}$, 
this implies that 
\begin{equation*}
(I+NN_{0}\mathcal{O}_{max})^{1} \subset 
{\SULN} \subset 
{\SUL} \subset 
\mathcal{O}_{max}^{1}. 
\end{equation*}
Since every maximal order of ${\MQ}$ is conjugate to $M_{2}({\Z})$, 
there exists $g\in {\rm GL}_{2}({\Q})$ such that 
\begin{equation*}
\Gamma(NN_{0}) \subset 
{\rm Ad}_{g}({\SULN}) \subset 
{\rm Ad}_{g}({\SUL}) \subset 
{\SL}. 
\end{equation*}
This proves our claim.

\subsection{Inductive step}\label{ssec: induction unitary}

Let $q\geq 2$. 
Suppose that Theorem \ref{thm: unitary} is proved for 
all Hermitian lattices of signature $(p', q')$ with $p'\leq q'<q$. 
We then prove the theorem for Hermitian lattices of signature $(p, q)$ with $p\leq q$. 
Since the argument is similar to the symplectic case, 
we will just indicate the outline. 
Let $I_{1}\ne I_{2}$ be two isotropic $K$-subspaces of ${\LK}$ 
of the same dimension, say $r$, 
and $Z_{1}, Z_{2}\subset {\X}$ the associated cusps. 
We make the following classification: 
\begin{enumerate}
\item $I_{1}\cap I_{2}\ne \{ 0 \}$;  
\item the pairing between $I_{1}$ and $I_{2}$ is perfect;  
\item $I_{1}\cap I_{2}= \{ 0 \}$ but the pairing between $I_{1}$ and $I_{2}$ is not perfect. 
\end{enumerate} 
  
(1) This is similar to \S \ref{ssec: symplectic case 1}. 
In this case $Z_{1}$ and $Z_{2}$ are joined by the cusp associated to 
$I_{1}\cap I_{2}$, to which we can apply the hypothesis of induction. 

(2) The case $r=1$ is similar to \S \ref{sssec: symplectic perfect k>0}. 
If we set $\Lambda_{K}'=I_{1}\oplus I_{2}$ and $\Lambda_{K}''=(\Lambda_{K}')^{\perp}$, 
these are nondegenerate of signature $(1, 1)$ and $(p-1, q-1)$ respectively. 
Then $Z_{1}$ and $Z_{2}$ are joined by 
the embedding 
$\mathcal{D}_{\Lambda'}\times \mathcal{D}_{\Lambda''} \hookrightarrow {\D}$. 
We can apply the Manin-Drinfeld theorem to $\mathcal{D}_{\Lambda'}$. 

The case $r>1$ is similar to \S \ref{sssec: symplectic perfect k=0}. 
We can interpolate $Z_{1}$ and $Z_{2}$ by a third cusp 
by taking a proper subspace $J_{1}\ne \{ 0\}$ of $I_{1}$ and 
setting $I_{3}=J_{1}\oplus (J_{1}^{\perp}\cap I_{2})$. 
Then we can apply the result of the case (1) to $I_{1}$ vs $I_{3}$ and to $I_{3}$ vs $I_{2}$. 

(3) This is similar to \S \ref{ssec: symplectic case 3}. 
We take splittings $I_{1}=J_{1}\oplus K_{1}$ and $I_{2}=J_{2}\oplus K_{2}$ 
such that $(J_{1}, I_{2})\equiv 0$, $(J_{2}, I_{1})\equiv 0$ and $(K_{1}, K_{2})$ perfect. 
We choose an isotropic subspace $J_{0}$ from $(K_{1}\oplus K_{2})^{\perp}$ 
with $(J_{1}, J_{0})$ and $(J_{2}, J_{0})$ perfect, 
and put $I_{3}=J_{0}\oplus K_{2}$ and $I_{4}=J_{0}\oplus K_{1}$. 
Then we apply 
the case (2) to $I_{1}$ vs $I_{3}$ and to $I_{4}$ vs $I_{2}$, 
and the case (1) to $I_{3}$ vs $I_{4}$ when $K_{i}\ne \{ 0 \}$. 
This proves Theorem \ref{thm: unitary}.

\begin{remark}\label{remark3} 
As in the symplectic case,  
we see that when $Z_{1}, Z_{2}$ are not top dimensional, 
their rational equivalence can be obtained through 
a chain of higher dimensional cusps of length $\leq 5$. 
\end{remark}

\section{Modular units and higher Chow cycles}\label{sec: higher modular unit}

Let ${\G}$, ${\D}$ and ${\X}$ be as in the previous sections. 
As a consequence of Theorems \ref{thm: orthogonal} -- \ref{thm: unitary}, 
we can associate to each pair of maximal cusps of ${\X}$ 
a nonzero higher Chow cycle of the modular variety ${\Y}={\G}\backslash{\D}$. 
This gives a higher dimensional analogue of modular units (\cite{K-L}) from the viewpoint of algebraic cycles. 

Let $Z_{1}\ne Z_{2}$ be two cusps of ${\X}$ of the same dimension, say $k$. 
By our result, we have $[Z_{1}]=\alpha [Z_{2}]$ in $CH_{k}({\X})_{{\Q}}$ for some 
$\alpha\ne 0 \in {\Q}$. 
On the other hand, we can also view $Z_{1}, Z_{2}$ as $k$-cycles on 
the boundary $\partial{\X}={\X}-{\Y}$, 
which is an equidimensional reduced closed subscheme of ${\X}$. 

\begin{lemma}\label{lem:(non)vanishing at boundary} 
When the cusps $Z_{1}, Z_{2}$ are not top dimensional,  
the equality $[Z_{1}]=\alpha[Z_{2}]$ holds already in $CH_{k}(\partial{\X})_{{\Q}}$. 
\end{lemma}

\begin{proof}
When $Z_{1}, Z_{2}$ are not top dimensional, 
the proof of Theorems \ref{thm: orthogonal} -- \ref{thm: unitary} 
and Remarks \ref{remark1}, \ref{remark2}, \ref{remark3} show that 
we can connect $Z_{1}$ and $Z_{2}$ by a chain of higher dimensional cusps. 
To be more precise, 
we have (congruence) modular varieties $X_{1}, \cdots, X_{N}$, 
their cusps $Z_{i}^{+}, Z_{i}^{-}\subset X_{i}$ of dimension $k$,  
and a finite morphism $f_{i}: X_{i}\to {\X}$ onto a cusp of ${\X}$, 
such that 
$f_{i}(Z_{i}^{-})=f_{i+1}(Z_{i+1}^{+})$ for each $i$ 
and $f_{1}(Z_{1}^{+})=Z_{1}$, $f_{N}(Z_{N}^{-})=Z_{2}$. 
By induction on dimension, we have 
$[Z_{i}^{+}]=\alpha_{i}[Z_{i}^{-}]$ in $CH_{k}(X_{i})_{{\Q}}$ 
for some $\alpha_{i}\in{\Q}$. 
Since $f_{i}$ factors through 
$X_{i}\to \partial X_{{\G}}\subset X_{{\G}}$, then  
$[f_{i}(Z_{i}^{+})]=\alpha_{i}'[f_{i}(Z_{i}^{-})]$ 
in $CH_{k}(\partial {\X})_{{\Q}}$ 
for some $\alpha_{i}'\in{\Q}$. 
It follows that 
$[Z_{1}]=(\prod_{i}\alpha_{i}')[Z_{2}]$ in $CH_{k}(\partial {\X})_{{\Q}}$. 
%
\end{proof}

Consider the localization exact sequence of higher Chow groups (\cite{Bl}, \cite{Bl2}) 
for the Baily-Borel compactification 
\begin{equation*}
{\Y} \stackrel{j}{\hookrightarrow} {\X} \stackrel{i}{\hookleftarrow} \partial {\X}. 
\end{equation*}
The first few terms of this sequence are written as   
\begin{equation*}
\cdots \to CH_{k}({\X}, 1)_{{\Q}} 
\stackrel{j^{\ast}}{\to} CH_{k}({\Y}, 1)_{{\Q}} \stackrel{\delta}{\to} CH_{k}(\partial{\X})_{{\Q}} 
\stackrel{i_{\ast}}{\to} CH_{k}({\X})_{{\Q}} \to \cdots, 
\end{equation*}
where $\delta$ is the connecting map. 
By Lemma \ref{lem:(non)vanishing at boundary}, 
the ${\Q}$-linear subspace of $CH_{k}(\partial{\X})_{{\Q}}$ 
generated by the $k$-dimensional cusps 
has dimension $1$ if $k$ is not the maximal dimension of cusps. 
On the other hand, when $k=\dim \partial{\X}$, 
the $k$-dimensional(=maximal) cusps are irreducible components of $\partial{\X}$, 
so $CH_{k}(\partial{\X})_{{\Q}}$ is freely generated over ${\Q}$ by those cusps. 
Let $t$ be the number of maximal cusps of ${\X}$. 
Since the image of 
$i_{\ast}\colon CH_{k}(\partial{\X})_{{\Q}} \to CH_{k}({\X})_{{\Q}}$ 
has dimension $1$ by Theorems \ref{thm: orthogonal} -- \ref{thm: unitary}, 
we find that 
\begin{equation*}
\dim {\rm Im} (\delta) \: = \: \dim {\rm Ker} (i_{\ast}) \: = \: t-1. 
\end{equation*}
Let us construct some explicit elements of 
$CH_{k}({\Y}, 1)_{{\Q}}$ 
whose image by $\delta$ generate 
${\rm Im} (\delta) = {\rm Ker} (i_{\ast})$. 


Let $Z_{1}\ne Z_{2}$ be two maximal cusps of ${\X}$, say of dimension $k=\dim \partial{\X}$.  
As above, we have 
$i_{\ast}(Z_{1}-\alpha Z_{2})=0$ in $CH_{k}({\X})_{{\Q}}$ 
for some $\alpha \in {\Q}$. 
We will construct an element of $CH_{k}({\Y}, 1)_{{\Q}}$ 
whose image by $\delta$ is $Z_{1}-\alpha Z_{2}$ in $CH_{k}(\partial{\X})_{{\Q}}$. 
(Such an element must be nonzero because 
$Z_{1}-\alpha Z_{2}$ is nonzero in $CH_{k}(\partial{\X})_{{\Q}}$.)  
Recall from the proof of Theorems \ref{thm: orthogonal} -- \ref{thm: unitary} that, 
in a basic case, 
we have  
a compactified modular curve $X'=X_{\Gamma'}$,  
its two cusps $p_{1}, p_{2}\in X'$, 
a $k$-dimensional compactified modular variety $X''=X_{\Gamma''}$, 
and a finite morphism $f\colon X'\times X'' \to {\X}$ 
such that $f(p_{i}\times X'')=Z_{i}$. 
(In the orthogonal case $X''$ is one point when $k=0$ and a modular curve when $k=1$; 
in the symplectic case $X''$ is a Siegel modular variety of genus $g-1$; 
in the unitary case $X''$ is associated to a unitary group of signature $(p-1, q-1)$.) 
The general case is a chain of such basic cases. 
For simplicity we assume that $(Z_{1}, Z_{2})$ is such a basic pair. 

By the Manin-Drinfeld theorem for $X'$, 
there exists a modular function $F$ on $X'$ such that 
${\rm div}(F)=\beta(p_{1}-p_{2})$ for some natural number $\beta$.   
Let $Y'\subset X'$ and $Y''\subset X''$ be the modular varieties before compactification. 
We can view $F$ as an element of $\mathcal{O}^{\ast}(Y')=CH_{0}(Y', 1)$. 
Then $\delta(F)=\beta(p_1 - p_2)$ for the connecting map 
$\delta\colon CH_{0}(Y', 1)\to CH_0(\partial X')$. 
Let $\pi:Y'\times Y''\to Y'$ be the projection and, 
by abuse of notation, 
$f:Y'\times Y''\to {\Y}$ be the restriction of $f:X'\times X''\to {\X}$. 
We can pullback the higher Chow cycle $F$ by the flat morphism $\pi$ 
and then take its pushforward by the finite morphism $f$. 
The result, $f_{\ast}\pi^{\ast}F$, 
is an element of $CH_{k}({\Y}, 1)$. 

\begin{proposition}
We have 
${\Q}\delta(f_{\ast}\pi^{\ast}F) = {\Q}(Z_{1}-\alpha Z_{2})$ 
in $CH_{k}(\partial{\X})_{{\Q}}$. 
\end{proposition}

\begin{proof}
We take a desingularization $\tilde{X}''\to X''$ of $X''$, 
and let $\tilde{Y}''\subset \tilde{X}''$ be the inverse image of $Y''$.  
We have the commutative diagram 
\begin{equation*}
\begin{CD}
\mathcal{O}^{\ast}(Y') @>\tilde{\pi}^{\ast}>\simeq> 
  \mathcal{O}^{\ast}(Y'\times\tilde{X}'') @>\tilde{j}^{\ast}>> 
\mathcal{O}^{\ast}(Y'\times\tilde{Y}'') @.  \\ 
@| @| @| @. \\ 
CH_{0}(Y', 1) @>\tilde{\pi}^{\ast}>\simeq>   CH_{k}(Y'\times\tilde{X}'', 1) @>\tilde{j}^{\ast}>> 
CH_{k}(Y'\times\tilde{Y}'', 1) @>\tilde{f}_{\ast}>> CH_{k}(Y_{{\G}}, 1) \\ 
@V\delta VV @V\delta VV @V\delta VV @V\delta VV \\ 
CH_{0}(\partial X') @>\tilde{\pi}^{\ast}>\simeq> CH_{k}(\partial X'\times \tilde{X}'') @>\tilde{i}_{\ast}>> 
CH_{k}(\partial(X'\times\tilde{X}'')) @>\tilde{f}_{\ast}>> CH_{k}(\partial X_{{\G}}). 
\end{CD}
\end{equation*}
Here various $\delta$ are the connecting maps of each localization sequence, 
$\tilde{\pi}\colon X'\times \tilde{X}''\to X'$ the projection,   
$\partial(X'\times\tilde{X}'')=X'\times \tilde{X}'' - Y'\times \tilde{Y}''$, 
$\tilde{j}\colon Y'\times\tilde{Y}'' \hookrightarrow Y'\times \tilde{X}''$ the open immersion,  
$\tilde{i}:\partial X' \times \tilde{X}'' \hookrightarrow \partial(X'\times \tilde{X}'')$ 
the closed embedding, and  
$\tilde{f}\colon X'\times \tilde{X}''\to {\X}$ the proper morphism induced from $f$. 
If we send ${\Q}F\subset CH_{0}(Y', 1)_{{\Q}}$ 
through this diagram to $CH_{k}(\partial{\X})_{{\Q}}$, 
the image is ${\Q}(Z_{1}-\alpha Z_{2})$. 
The assertion follows by noticing that 
$\tilde{f}_{\ast}\tilde{j}^{\ast}\tilde{\pi}^{\ast} = f_{\ast} \pi^{\ast}$. 
\end{proof}

In this way, as a ``lift'' from the modular unit $F$, 
we obtain an explicit nonzero element of $CH_{k}({\Y}, 1)_{{\Q}}$ 
whose image by $\delta$ is $Z_1-\alpha Z_2$. 
If we run $(Z_{1}, Z_{2})$ over all basic pairs of maximal cusps, 
we obtain a set of nonzero elements of $CH_{k}({\Y}, 1)_{{\Q}}$ 
whose image by $\delta$ generate ${\rm Im}(\delta)={\rm Ker}(i_{\ast})$. 
In general, by this construction 
we could obtain more than $t-1$ higher Chow cycles on ${\Y}$. 
This is because    
\begin{enumerate}
\item the choice of $X'\times X'' \to {\X}$ is not necessarily unique for the given pair $(Z_{1}, Z_{2})$, and 
\item the number of basic pairs could be larger than $t-1$. 
\end{enumerate}
The point (1) amounts to the situation that 
two pairs $(I_{1}, I_{2})$, $(I_{1}', I_{2}')$ of isotropic subspaces 
are not ${\G}$-equivalent as pairs 
although $I_{1}$ is ${\G}$-equivalent to $I_{1}'$ and 
$I_{2}$ is ${\G}$-equivalent to $I_{2}'$ respectively. 
A typical situation of (2) is that 
for three cusps $Z_{1}, Z_{2}, Z_{3}$, 
all pairs $(Z_{1}, Z_{2})$, $(Z_{2}, Z_{3})$, $(Z_{3}, Z_{1})$ are basic. 
 
If the span $V\subset CH_{k}({\Y}, 1)_{{\Q}}$ of all higher Chow cycles 
constructed in this way  
has dimension $\geq t$, 
the kernel of $\delta \colon V\to CH_{k}(\partial{\X})_{{\Q}}$ 
would then give rise to a nontrivial subspace of $CH_{k}({\X}, 1)_{{\Q}}$. 



\end{document}